\documentclass[12pt]{amsart}
\usepackage{amsmath}
\usepackage{amscd}
\usepackage{amssymb}
\usepackage{amsfonts}
\usepackage{graphicx}
\usepackage{wrapfig}
\newtheorem{thm}{Theorem}

\newtheorem{theorem}{Theorem}[section]
\newtheorem{lemma}[theorem]{Lemma}
\newtheorem{corollary}[theorem]{Corollary}
\newtheorem{proposition}[theorem]{Proposition}

\theoremstyle{definition}

\theoremstyle{remark}
\newtheorem{remark}[theorem]{Remark}
\numberwithin{equation}{section}
\newcommand{\abs}[1]{\lvert#1\rvert}

\usepackage[utf8]{inputenc}
\usepackage{tikz}
\usetikzlibrary{patterns}
\pgfdeclarepatternformonly{my dots}{\pgfqpoint{-1pt}{-1pt}}{\pgfqpoint{2.5pt}{2.5pt}}{\pgfqpoint{6pt}{6pt}}%
{
	\pgfpathcircle{\pgfqpoint{0pt}{0pt}}{.5pt}
	\pgfpathcircle{\pgfqpoint{7pt}{7pt}}{.5pt}
	\pgfusepath{fill}
}
\begin{document}
\title[Gradient estimates for weighted harmonic function]
{Gradient estimates for weighted harmonic function with Dirichlet boundary condition}

\author{Nguyen Thac Dung}
\address{Department of Mathematics, Vietnam National University (VNU), University of Science,
Hanoi, Vietnam}
\email{dungmath@vnu.edu.vn}

\author{Jia-Yong Wu}
\address{Department of Mathematics, Shanghai University, Shanghai 200444, China}
\email{wujiayong@shu.edu.cn}

\thanks{}
\dedicatory{}
\date{\today}


\begin{abstract}
We prove a Yau's type gradient estimate for positive $f$-harmonic functions with
the Dirichlet boundary condition on smooth metric measure spaces with compact boundary
when the infinite dimensional Bakry-Emery Ricci tensor and the weighted mean curvature are
bounded below. As an application, we give a Liouville type result for bounded $f$-harmonic
functions with the Dirichlet boundary condition. Our results do not depend on any assumption
on the potential function $f$.
\end{abstract}
\maketitle

\section{Introduction}\label{Int1}
In this paper, we will give Yau's type gradient estimates for positive $f$-harmonic functions
with the Dirichlet boundary condition (i.e., they are constant on the boundary) on smooth metric measure
spaces with the compact boundary when the infinite dimensional Bakry-\'Emery Ricci tensor and
the weighted mean curvature are bounded below. As an application, we will prove a Liouville
theorem for bounded $f$-harmonic functions with the Dirichlet boundary condition.

Recall that an $n$-dimensional smooth metric measure space denoted by $(M,g,e^{-f}dv_g)$ is an
$n$-dimensional smooth complete Riemannian manifold $(M,g)$ coupled with a weighted volume
$e^{-f}dv_g$ for some $f\in C^\infty(M)$, where $dv_g$ is the standard Riemannian volume
element on $M$ and $f$ is called the potential function. Smooth metric measure spaces are
closely related to gradient Ricci solitons,
the Ricci flow, probability theory, and optimal transport; see e.g. \cite{[BE]}, \cite{[P1]}
and \cite{LV}. On smooth metric measure space $(M,g,e^{-f}dv_g)$, for any $m>0$, the
$m$-Bakry-\'Emery Ricci tensor, introduced by Bakry and \'Emery \cite{[BE]}, is defined by
\[
\mathrm{Ric}_f^m:=\mathrm{Ric}+\mathrm{Hess}\,f-\frac{1}{m}df\otimes df,
\]
where $\mathrm{Ric}$ is the Ricci tensor of the manifold $(M,g)$ and $\mathrm{Hess}$ is the
Hessian with respect to the Riemannian metric $g$. Clearly, $m$-Bakry-\'Emery Ricci tensor is
a natural generalization of Ricci curvature on Riemannian manifolds.

When $m<\infty$, $\mathrm{Ric}_f^m$ is called the finite dimensional Bakry-\'Emery Ricci
tensor and it often shares many similar geometric results for $(n+m)$-dimensional
manifolds with the Ricci tensor; see for example \cite{[LD]}, Appendix A in \cite{[WW]}
and references therein. This is because the Bochner formula for $\mathrm{Ric}_f^m$ can
be considered as the Bochner formula for the Ricci tensor of an $(n+m)$-dimensional
manifold, i.e.,
\begin{equation*}
\begin{split}
\frac 12\Delta_f|\nabla u|^2=&|\nabla^2
u|^2+\langle\nabla\Delta_f u, \nabla u\rangle+\mathrm{Ric}_f^m(\nabla u, \nabla u)
+\frac 1m|\langle\nabla f,\nabla u\rangle|^2\\
\geq& \frac{(\Delta_f u)^2}{m+n}+\langle\nabla\Delta_f u, \nabla
u\rangle+\mathrm{Ric}_f^m(\nabla u, \nabla u)
\end{split}
\end{equation*}
for any $u\in C^\infty(M)$, where $\Delta_f$ is called the $f$-Laplacian, which is defined by
\[
\Delta_f:=\Delta-\nabla f\cdot\nabla.
\]
This operator is a natural generalization of the usual Laplacian and is self-adjoint with respect to the weighted measure $e^{-f}dv_g$.
When $m=\infty$, we have
\[
\mathrm{Ric}_f:=\lim_{m\to\infty}\mathrm{Ric}_f^m=\mathrm{Ric}+\nabla^2 f,
\]
which is called the infinite dimensional Bakry-\'Emery Ricci tensor. When $\mathrm{Ric}_f$ is
bounded below, many geometric properties of manifolds with the Ricci tensor bounded below
were also possibly generalized to smooth metric measure spaces; but some extra assumption
on $f$ is needed, see for example \cite{[WW]}, \cite{[WW1]}, \cite{[WW2]} and references
therein. It is easy to see that $\mathrm{Ric}_f^m\ge c$ implies $\mathrm{Ric}_f\ge c$,
but the opposite may be not true.

In particular, if there exists a real constant $\lambda$ such that
\[
\mathrm{Ric}_f=\lambda g,
\]
then $(M,g,e^{-f}dv_g)$ is called the gradient Ricci soliton.
The gradient Ricci soliton is called shrinking, steady, or expanding, if $\lambda>0$,
$\lambda=0$, or $\lambda<0$, respectively. Gradient Ricci solitons are natural
generalizations of Einstein metrics. They are also self-similar solutions to the
Ricci flow and play important roles in the Ricci flow and Perelman's resolution of
the Poincar\'e conjecture and the geometrization conjecture; see \cite{[Ham]},
\cite{[P1]}, \cite{[P2]}, \cite{[P3]} and references therein for nice details.

On smooth metric measure space $(M,g,e^{-f}dv_g)$, a smooth function $u$ is called
$f$-harmonic (also called weighted harmonic) if
\[
\Delta_f u=0.
\]
On $(M,g,e^{-f}dv_g)$ with compact boundary $\partial M$, the $f$-mean
curvature (also called weighted mean curvature) is defined by
\[
\mathrm{H}_f:=\mathrm{H}-\nabla f\cdot\nu,
\]
where $\nu$ is the unit outer normal vector to $\partial M$ and $\mathrm{H}$ is the mean curvature of
$\partial M$ with respect to $\nu$. When $f$ is constant, the above concepts all recover
the manifold case.

For manifolds with the boundary, most of geometric results concentrate on the Neumann boundary
condition; see for example \cite{Ch}, \cite{WaJ}, \cite{[LY]} and \cite{Ol}. Recently,
Kunikawa and Sakurai \cite{KS20} extended Yau's gradient estimates and Liouville theorems
for harmonic functions \cite{Yau} to the Dirichlet boundary condition. Shortly after, H. Dung,
N. Dung and Wu \cite{DDW} further extended their results to $f$-harmonic functions on the
smooth metric measure space with some compact boundary. In summary, we have
\begin{thm}\label{ThmA}
Let $(M,g,e^{-f}dv)$ be an $n$-dimensional smooth metric measure space with the compact boundary.
For a fixed $m>0$, assume that
\[
{\rm Ric}_f^m\ge-(n+m-1)K\quad \mathrm{and} \quad \mathrm{H}_f\ge-L
\]
for some non-negative constants $K$ and $L$. Let
$u: B_R(\partial M)\to (0, \infty)$ be a positive $f$-harmonic function
with the Dirichlet boundary condition. If $u_\nu$ is non-negative over $\partial M$, then
there exists a constant $c$ depending on $n+m$ such that
\[
\sup\limits_{B_{R/2}(\partial M)}\frac{|\nabla u|}{u}\le
c\left(\frac{1}{R}+L+\sqrt{K}\right).
\]
Here $B_R(\partial M):=\{x\in M|d(x,\partial M)<R\}$.
\end{thm}
In Theorem \ref{ThmA}, ${\rm Ric}_f^m\ge-(n+m-1)K$ means that the infimum of ${\rm Ric}_f^m$
on the unit tangent bundle on the interior of $M$ is at least $-(n+m-1)K$; $\mathrm{H}_f\ge-L$
means that boundary $\partial M$ has some weak convex property. As pointed out in
\cite{DDW}, when ${\rm Ric}_f\ge-(n-1)K$ and $\mathrm{H}_f\ge-L$, there seems to be essential
obstacles to derive Yau's type gradient estimates by directly following their proof of Theorem
\ref{ThmA} in \cite{KS20} or \cite{DDW}. This is because their proof depends on a refined Kato
inequality, which is not suitable to the case when
${\rm Ric}_f\ge-(n-1)K$.

In this paper, we will solve the above question and give Yau's type gradient estimates
for positive $f$-harmonic functions with the Dirichlet boundary on a neighborhood of the
boundary under lower bounds of ${\rm Ric}_f$ and $\mathrm{H}_f$. Our gradient estimate
does not depend on any assumption on $f$.
\begin{theorem}\label{main}
Let $(M,g,e^{-f}dv)$ be an $n$-dimensional smooth metric measure space with the compact boundary.
Assume that
\[
{\rm Ric}_f\ge-(n-1)K\quad \mathrm{and} \quad \mathrm{H}_f\ge-L
\]
for some non-negative constants $K$ and $L$. Let $u$ be a positive $f$-harmonic function on
$B_R(\partial M)$ with the Dirichlet boundary condition. If $u_\nu\ge 0$ over $\partial M$,
then there exists a constant $c(n)$ depending on $n$ such that
\[
\sup\limits_{B_{R/2}(\partial M)}|\nabla u|\le
c(n)\left(\frac{1}{R}+L+\sqrt{K}\right)\sup_{y\in B_R(\partial M)}u(y).
\]
\end{theorem}

Notice that our gradient estimate in Theorem \ref{main} holds for
all $R>0$; however Brighton's result (see Theorem 1 in \cite{Br}) in the complete
non-compact case without the boundary only holds for $R>1$. The reason is that in our
case we use a new weighted Laplacian comparison on any neighborhood of the boundary;
see Theorem \ref{Lapcom} in Section \ref{sec2}. Also notice that if $d(x,\partial M)$,
where $x\in M$, is finite, then all boundary balls $B_R(\partial M)$ are possibly the
same if $R$ is large.

Li and Yau \cite{[LY]} proved gradient estimates for the heat equation with the Neumann boundary
condition on compact manifolds of the convex boundary. From the course of our proof, it is easy
to see that our gradient estimate also holds for $f$-harmonic functions with the Neumann boundary condition.

The main trick of proving Theorem \ref{main} stems from the arguments of Brighton \cite{Br}
and Kunikawa and Sakurai \cite{KS20}. Far away from the boundary of space, we
will apply Yau's gradient estimate technique to function $u^{7/8}$ instead of $\ln u$.
On the boundary of space, we will apply a derivative equality (Proposition
\ref{Reilly} in Section \ref{sec2}) to prove the desired estimate.

Taking a limit as $R$ tends to infinity, we immediately get a Liouville type theorem for
$f$-harmonic functions with the Dirichlet boundary condition.
\begin{corollary}\label{Liou}
Let $(M,g,e^{-f}dv)$ be an $n$-dimensional smooth metric measure space with compact boundary $\partial M$.
If ${\rm Ric}_f\ge 0$ and $\mathrm{H}_f\ge 0$, then any bounded $f$-harmonic function $u$ with
the Dirichlet boundary condition and $u_\nu\ge 0$ over $\partial M$ must be constant.
\end{corollary}
\begin{remark}
The assumption $u_\nu\ge 0$ over $\partial M$ in Corollary \ref{Liou} is necessary. For example, let
$u(x)=x$ and $f=\mathrm{constant}$ in $M=[1,2]$ with compact boundary $\{1,2\}$.
Then,
\[
{\rm Ric}_f=\mathrm{H}_f=\Delta_fu=0
\]
However, $u_\nu|_{x=1}=-1$ and $u_\nu|_{x=2}=1$. Thus $u(x)$ is a non-constant
bounded $f$-harmonic function with the Dirichlet boundary condition.
\end{remark}

\begin{remark}
The boundness of $f$-harmonic function in Corollary \ref{Liou} is necessary. Indeed, let
$u(x)=-e^{-x}$ and $f(x)=-x$ in $M=(-\infty,0]$ with compact boundary $\{0\}$. Then
the unit outer normal vector $\nu=1$,
\[
{\rm Ric}_f=0,\quad \mathrm{H}_f=1, \quad \Delta_f u=u''+u'=0 \quad  \mathrm{and} \quad u_\nu|_{x=0}=1.
\]
However, $u(x)=-e^{-x}$ is unbounded in $(-\infty,0]$ and it is a non-constant $f$-harmonic
function with the Dirichlet boundary condition.
\end{remark}

\begin{remark}
The assumption $\mathrm{H}_f\ge 0$ in Corollary \ref{Liou} is necessary. We provide two
examples to illustrate it. One example is that, for any real number $\alpha>0$, let
$u(x)=e^{\alpha x}$ and $f(x)=\alpha x$ in $M=(-\infty,0]$ with compact boundary
$\{0\}$. Then $\nu=1$,
\[
{\rm Ric}_f=0, \quad \Delta_f u=u''-\alpha u'=0 \quad  \mathrm{and} \quad u_\nu|_{x=0}=\alpha.
\]
However, $\mathrm{H}_f=-\alpha<0$ and $u(x)$ is a non-constant bounded $f$-harmonic function with
the Dirichlet boundary condition.

Another example is that $u(x)=x^{-1}$ and $f(x)=-2\ln x$ in $M=[1,\infty)$ with
compact boundary $\{1\}$. Then $\nu=-1$,
\[
{\rm Ric}_f=2x^{-2}\ge0, \quad \Delta_fu=0 \quad  \mathrm{and} \quad u_\nu|_{x=1}=1.
\]
However, $\mathrm{H}_f=-2x^{-1}<0$
and $u(x)$ is a non-constant bounded $f$-harmonic function with the Dirichlet
boundary condition.
\end{remark}

The rest of this paper is organized as follows. In Section \ref{sec2}, we will recall
some results about smooth metric measure spaces with the compact boundary, including
the weighted Laplacian comparison, the derivative equality, the Bochner type formula
and the cut-off function. These results will be used in the proof of our gradient
estimate. In Section \ref{sec3}, we will apply the Brighton's trick \cite{Br} and the
Kunikawa-Sakurai's argument \cite{KS20} to prove Theorem \ref{main}.

\textbf{Acknowledgement}.
The authors would like to thank the referee for valuable comments and useful suggestions for this work. The first author is supported by the research project QG.21.01 ``Geometric operators on Riemannian manifolds" of Vietnam National University, Hanoi. The second author is supported by the Natural Science Foundation of Shanghai (17ZR1412800).


\section{Background}\label{sec2}
In this section, we list some known results about smooth metric measure spaces with the
boundary. These results will be used in the proof of our result. For more properties, the
interested reader are referred to \cite{Sakurai17}. On a smooth metric measure space
$(M^n, g, e^{-f}dv)$ with the boundary $\partial M$, the distance function from
the boundary is denoted by
\[
\rho(x)=\rho_{\partial M}(x)=d(x, \partial M),
\]
where $x\in M$.
By \cite{Sa17}, we may assume that $\rho$ is smooth outside of the cut locus for the
boundary. In \cite{WZZ},  Wang, Zhang and Zhou obtained weighted Laplacian comparisons
for the distance function on smooth metric measure spaces with the boundary under some
assumptions on $\rho(x)$ (see also \cite{Saku}). Later, Sakurai \cite{Sakurai17} proved the
following general comparison result without any assumption on $\rho(x)$, which is a key step
in our proof of Theorem \ref{main}.
\begin{theorem}\label{Lapcom}
Let $(M^n, g, e^{-f}dv)$ be an $n$-dimensional complete smooth metric measure space with compact boundary
$\partial M$. Assume that
\[
{\rm Ric}_f\ge-(n-1)K\quad \mathrm{and} \quad \mathrm{H}_f\ge-L
\]
for some constants $K\ge0$ and $L\in\mathbb{R}$. Then
\[
\Delta_f\rho(x)\le (n-1)KR+L
\]
for all $x\in B_R(\partial M)$.
\end{theorem}
\begin{proof}[Proof of Theorem \ref{Lapcom}]
We will give a quick explanation of the result based on Sakurai's result \cite{Sakurai17}.
Assume that $x\in B_{R}(\partial M)$. Let $B(x, d(x, \partial M))\subset M$ be the largest geodesic ball
with center $x$ such that $\partial B(x, d(x, \partial M))\cap \partial M=z$. We also let
$\gamma_{z, x}(s)$ be a geodesic line with the arc-length parameter $s$ which starts from point
$z$ to $x$. It is easy to see that $\gamma'_{z,x}(0)$ is the unit inner normal vector for $\partial M$ at $z$.
Since $\gamma_{z,x}(d(x, \partial M))=x$ and $d(x, \partial M)\leq R$, by Lemma 6.1 in \cite{Sakurai17},
we easily get Theorem \ref{Lapcom}. We remark that the above statement can be described below; see Figure 1.

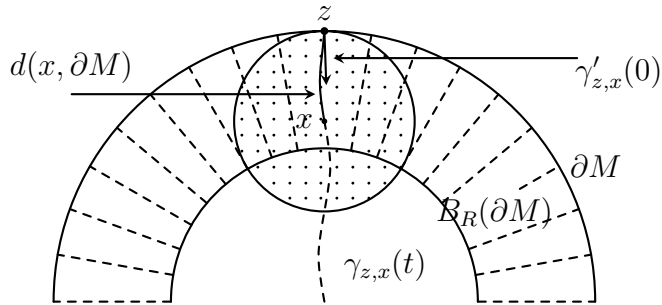
\begin{figure}[htb]
\begin{tikzpicture}[>=stealth,line join=round,line cap=round,thick]
	\tikzset{label style/.style={font=\footnotesize},scale=1.2}
	\fill[pattern=my dots](90:2) circle (1) ;
	\draw (1.7,0) arc (0:180:1.7)	(3,0) arc (0:180:3)
	(90:2) circle (1)	(0,2)..controls +(100:0.5) and +(-95:0.5)..(0,3);
	\draw[->] (0,3)--(89.5:2.4);
	\draw[<-] (0.1,2.7)--(2.8,2.7)node[below right=-5pt]{$\gamma'_{z,x}(0)$};
	\draw[->] (-2.8,2.3)node[above]{$d(x,\partial M)$} --(-0.1,2.3);	
	\draw[dashed]
	(0,0) ..controls +(100:0.5) and +(-100:0.5)..(0,1)
	..controls +(80:0.5) and +(-80:0.5)..(0,2);	
	\filldraw[black] (90:3) circle (0.9pt) node[above] {$z$};
	\fill[black]	(90:2) circle (0.9pt)  node[left]{$x$};
	\draw (30:3) node[right]{$\partial M$}
	(26:3) node[below left=1pt] {$B_R(\partial M)$}
	(0,0) node[above right=3pt]{$\gamma_{z,x}(t)$};	
	\foreach \i in {0,10,20,40,50,60,70,80,100,110,120,130,140,150,160,170,180}
    \draw[dashed] (\i:1.7)--(\i:3);
	\foreach \i in {30} \draw[dashed] (\i:2.5)--(\i:3);
\end{tikzpicture}
\caption{Comparison on weighted manifolds with boundary}
\end{figure}
\end{proof}
Next, we recall the following derivative equality, which was ever used in the proof of the weighted
Reilly formula in \cite{[MD]}; see also (24) in Appendix of \cite{[CMZ]}. In fact it is
a slight generalization of the classical case \cite{Rei}. This formula will be used in our
gradient estimate for the boundary of the weighted manifold.
\begin{proposition}\label{Reilly}
Let $(M, g, e^{-f}dv)$ be a complete smooth metric measure space with compact boundary
$\partial M$. For any $u\in C^\infty(M)$, we have
\begin{equation*}
\begin{aligned}
\frac{1}{2}\left(|\nabla u|^2\right)_\nu&=u_\nu\left[\Delta_f u-\Delta_{\partial M, f}(u|_{\partial M})
-\mathrm{H}_f u_\nu\right]
+g_{\partial M}(\nabla_{\partial M}(u|_{\partial M}), \nabla_{\partial M}u_\nu)\\
&\quad-{\rm II}(\nabla_{\partial M}(u|_{\partial M}), \nabla_{\partial M}(u|_{\partial M})),
\end{aligned}
\end{equation*}
where $\nu$ is the outer unit normal vector to $\partial M$, and ${\rm II}$ is the second
fundamental form of $\partial M$ with respect to $\nu$.
\end{proposition}
\begin{proof}[Proof of Proposition \ref{Reilly}]
Its proof follows by a direct computation and we include it for the sake of completeness.
We compute that
\begin{equation*}
\begin{aligned}
\frac{1}{2}\left(|\nabla u|^2\right)_\nu&=g(\nabla_{\nu}\nabla u,\nabla u)\\
&=g(\nabla_{\nabla u}\nabla u,\nu)\\
&=g(\nabla_{\nu}\nabla u,\nu)u_{\nu}+g(\nabla_{\nabla_{\partial M} u}\nabla u,\nu)\\
&=g(\nabla_{\nu}\nabla u,\nu)u_{\nu}+g_{\partial M}(\nabla_{\partial M}u, \nabla_{\partial M}{u_\nu})
-g(\nabla u,\nabla_{\nabla_{\partial M}u}\nu).
\end{aligned}
\end{equation*}
Thus,
\begin{equation*}
\begin{aligned}
\frac{1}{2}&\left(|\nabla u|^2\right)_\nu-(\Delta_f u)u_{\nu}\\
&=\big[g(\nabla_{\nu}\nabla u,\nu)-\Delta u+g(\nabla f,\nabla u)\big]u_{\nu}
+g_{\partial M}(\nabla_{\partial M}u, \nabla_{\partial M}{u_\nu})
-g(\nabla_{\partial M} u,\nabla_{\nabla_{\partial M}u}\nu)\\
&=\big[-\Delta_{\partial M}(u|_{\partial M})-\mathrm{H}u_{\nu}+g_{\partial M}(\nabla_{\partial M} (f|_{\partial M}),\nabla_{\partial M}(u|_{\partial M}))+g(\nabla f,\nu)u_{\nu}\big]u_{\nu}\\
&\quad+g_{\partial M}(\nabla_{\partial M}u, \nabla_{\partial M}{u_\nu})
-{\rm II}(\nabla_{\partial M}(u|_{\partial M}), \nabla_{\partial M}(u|_{\partial M}))
\end{aligned}
\end{equation*}
and the desired result follows.
\end{proof}

Meanwhile, we recall an important Bochner type formula in the proof of our result, which was
proved by Brighton; see (2.10) in \cite{Br}. One important step of his proof is examining
two cases depending on the relative magnitudes between $\langle\nabla h,\nabla f\rangle$
and $\frac{|\nabla h|^2}{2h}$, where $h:=u^{7/8}$ and $u$ is a positive $f$-harmonic function.
We mention that this formula always holds without any assumption on $f$.
\begin{lemma}
Let $(M^n, g, e^{-f}dv)$ be an $n$-dimensional smooth metric measure space with
\[
{\rm Ric}_f\ge-(n-1)K
\]
for some constant $K\ge0$. If $u$ is a positive $f$-harmonic function on $M$, then
function $h:=u^{7/8}$ satisfies
\begin{equation}\label{evolu}
\frac 12\Delta_f|\nabla h|^2 \ge\frac{7n-6}{49nh^2}|\nabla h|^4-\frac{1}{7h}\langle\nabla h,\nabla|\nabla h|^2\rangle-(n-1)K|\nabla h|^2.
\end{equation}
\end{lemma}

In the end, we introduce a well-known cut-off function originated by Li and Yau \cite{[LY]}.
Here we may adopt the statements of \cite{[SZ]} and \cite{KS20}. The cut-off function
is an important tool in our proof of Yau's type gradient estimates.
	\begin{lemma}\label{lemmcut}
Let $(M^n, g, e^{-f}dv)$ be an $n$-dimensional complete smooth metric measure space
with compact boundary $\partial M$. There exists a smooth cut-off function $\phi=\phi(x)$
supported in $B_R(\partial M)$ such that
\begin{itemize}
\item[(i)] $\phi =\phi(\rho_{\partial M}(x))\equiv \phi(\rho)$;\,
$\phi(\rho)=1$ in $B_{R/2}(\partial M)$,\quad $0\le\phi\le 1$.

\item[(ii)] $\phi$ is decreasing as a radial function of parameter $\rho$.

\item[(iii)]$\left|\frac{\partial\phi}{\partial \rho}\right|\le \frac{C_{\varepsilon}{\phi^{\varepsilon }}}{R}$
and $\left| \frac{{\partial}^2\phi}{{\partial}{\rho^2}}\right|\le\frac{C_{\varepsilon}\phi^{\varepsilon}}{R^2}$,
\quad $0<\varepsilon <1$.
		\end{itemize}
	\end{lemma}


\section{Gradient estimate}\label{sec3}
In this section, we will combine the arguments of Kunikawa-Sakurai \cite{KS20}
and Brighton \cite{Br} to prove our result.

\begin{proof}[Proof of Theorem \ref{main}]
Let $u$ be a positive $f$-harmonic function on $M$ and let $h:=u^{7/8}$. We then
consider function
\[
G:=\phi|\nabla h|^2,
\]
where $\phi$ is a smooth cut-off function supported in
$B_R(\partial M)$ introduced in Lemma \ref{lemmcut}.  We would like to point out that
if $\rho$ is finite, then all boundary balls $B_R(\partial M)$ are possibly the
same if $R$ is large enough. We now compute that
\[
\nabla G=|\nabla h|^2\cdot\nabla\phi+\phi\cdot\nabla|\nabla h|^2.
\]
Hence
\begin{equation} \label{gradsquared}
\abs{\nabla h}^2 = \frac{G}{\phi}
\quad \text{and} \quad
\nabla \abs{\nabla h}^2 = \frac{\nabla G}{\phi} - \frac{\nabla \phi}{\phi ^2}G.
\end{equation}
We further compute that
\begin{equation*}
\begin{aligned}
\Delta_f G&=\Delta\phi|\nabla h|^2+2\langle\nabla\phi,\nabla|\nabla h|^2\rangle
+\phi\Delta|\nabla h|^2-\langle\nabla f,\nabla\phi\rangle|\nabla h|^2-\phi\langle\nabla f,\nabla|\nabla h|^2\rangle\\
&=\Delta\phi|\nabla h|^2+2\langle\nabla\phi,\nabla|\nabla h|^2\rangle
-\langle\nabla f,\nabla \phi\rangle|\nabla h|^2+\phi\Delta_f|\nabla h|^2.
\end{aligned}
\end{equation*}
From this, we get that
\begin{equation}
\begin{aligned}\label{flaplace}
\Delta_f|\nabla h|^2&=\frac{1}{\phi}\Delta_f G-\frac{|\nabla h|^2}{\phi}\Delta_f\phi- 2\left\langle\frac{\nabla \phi}{\phi},\nabla\abs{\nabla h}^2\right\rangle\\
&=\frac{1}{\phi}\Delta_f G-\frac{G}{\phi^2}\Delta_f\phi
-2\left\langle\frac{\nabla\phi}{\phi},\frac{\nabla G}{\phi}-\frac{\nabla\phi}{\phi^2}G\right\rangle.
\end{aligned}
\end{equation}
Then substituting \eqref{gradsquared}  and \eqref{flaplace} into \eqref{evolu} and
solving for $\frac{1}{\phi} \Delta_f G$, we obtain
\begin{equation}
\begin{aligned}\label{estima}
\frac{1}{\phi}\Delta_fG&\ge\frac{G}{\phi^2}\Delta_f\phi+2\left\langle\frac{\nabla\phi}{\phi},\frac{\nabla G}{\phi}\right\rangle
-2\left\langle\frac{\nabla \phi}{\phi},\frac{\nabla \phi}{\phi^2}G\right\rangle
+\left(\frac{14n-12}{49nh^2}\right) \frac{G^2}{\phi^2}\\
&\quad-\frac{2}{7h}\left\langle\nabla h,\frac{\nabla G}{\phi}-\frac{\nabla\phi}{\phi^2}G\right\rangle-2(n-1)K\frac{G}{\phi}.
\end{aligned}
\end{equation}

Now we assume that $G$ obtains its maximal value at $x_1\in B_R(\partial M)$. We will
prove the desired estimate according to two cases: $x_1\in B_R(\partial M)\setminus\partial M$
and $x_1\in\partial M$.

\textbf{Case 1: }If $x_1\in B_R(\partial M)\setminus\partial M$, we
may assume that $x_1\notin {\rm Cut}(\partial M)$ by the Calabi's argument.
At $x_1$, we have
\[
\nabla G=0\quad \mathrm{and}\quad \Delta_fG\le0.
\]
Consequently, at $x_1$, \eqref{estima} becomes
\begin{equation}\label{almostdone}
\left(\frac{14n-12}{49nh^2}\right)G\le-\Delta_f\phi+\frac{2}{\phi}|\nabla\phi|^2
-\frac{2}{7h}\langle\nabla h,\nabla\phi\rangle+2(n-1)\phi K.
\end{equation}
Since  ${\rm Ric}_f\ge-(n-1)K$ and $H_f\ge-L$,
by the weighted Laplacian comparison of Theorem \ref{Lapcom}, we have that
\begin{equation}
\begin{aligned}\label{dela}
\Delta_f\phi&\ge\frac{-c}{R}\Delta_f(\rho)-\frac{c}{R^2}\\
&\ge-c(n-1)K-\frac{-cL}{R}-\frac{c}{R^2}.
\end{aligned}
\end{equation}
For any $\delta\ge 0$, the term $\langle\nabla h, \nabla \phi\rangle$ can be estimated by
\begin{equation}
\begin{aligned}\label{cterm}
2\abs{\langle\nabla h,\nabla \phi\rangle}&\le 2\abs{\nabla h} \abs{\nabla \phi}\\
 &\le \frac{\delta \phi}{h}|\nabla h|^2+\frac{h}{\delta\phi}|\nabla\phi|^2.
\end{aligned}
\end{equation}
Substituting \eqref{dela} and \eqref{cterm} into \eqref{almostdone}, we have that
at $x_1$,
\begin{equation}
\left(\frac{14n-12}{49nh^2}-\frac{\delta}{7h^2}\right)G\le\frac{c L}{R}+\frac{c}{R^2}
+\left(2+\frac{1}{7\delta}\right)\frac{|\nabla\phi|^2}{\phi}+2(n-1)K(\phi+c) .
\end{equation}
Now choosing $\delta=\frac{1}{7}$, and using
\[
\frac{\abs{\nabla \phi}^2}{\phi} \leq \frac{c^2}{R^2},
\]
we get that
\begin{equation}
\begin{aligned}\label{maxest}
\frac{13n-12}{49nh(x_1)^2}G(x_1)&\le\frac{cL}{R}+\frac{c+3c^2}{R^2}+2(n-1)K(\phi(x_1)+c)\\
&\le\frac{cL}{R}+\frac{c+3c^2}{R^2}+2(n-1)K(1+c).
\end{aligned}
\end{equation}
Notice that
\[
\sup\limits_{B_{R/2}(\partial M)}G(x)\le G(x_1).
\]
Combining this with \eqref{maxest} and using the definition of $h$,  we conclude that
\[
\sup\limits_{B_{R/2}(\partial M)}|\nabla u|\le \sqrt{\frac{64n(cLR+c+3c^2)}{13n-12}\frac{1}{R^2} +\frac{128n(n-1)(1+ c)}{13n-12}K} \sup_{y\in B_R(\partial M)}u(y).
\]
So the desired result follows by using the Cauchy-Schwarz inequality
\[
2\sqrt{\frac{L}{R}}\le \frac{1}{R}+L.
\]

\textbf{Case 2:} If maximal point $x_1\in\partial M$, our gradient estimate still holds by adapting
the argument of \cite{KS20}. Indeed, at $x_1$, we have $G_{\nu}\ge0$. Since $\phi(x_1)=1$, then
\[
(|\nabla h|^2)_{\nu}\ge 0.
\]
Here we notice that the Dirichlet boundary condition for $u$ and the assumption $u_\nu\geq0$
imply that
\[
|\nabla u|=u_\nu.
\]
Since $h=u^{7/8}$ and $u$ satisfies the Dirichlet boundary
condition, by Proposition \ref{Reilly}, we have
\begin{align*}
0\le\frac{64}{49}\left(|\nabla h|^2\right)_\nu&=(u^{-1/4}|\nabla u|^2)_{\nu}\\
&=u^{-1/4}(|\nabla u|^2)_\nu-\frac 14u^{-5/4}|\nabla u|^2 u_{\nu}\\
&=2u^{-1/4}u_\nu(\Delta_f u-\mathrm{H}_f u_\nu)-\frac 14u^{-5/4}u^3_{\nu},
\end{align*}
where we used $|\nabla u|=u_\nu$ in the last line.
Since $u$ is a positive $f$-harmonic function, the above inequality reduces to
\[
\frac{u_\nu}{8u}+\mathrm{H}_f\le 0.
\]
By our theorem assumption, this implies
\[
\frac{u_\nu}{8u}\le-\mathrm{H}_f\le L.
\]
Hence,
\begin{equation*}
\begin{aligned}
G(x_1)=|\nabla h|^2(x_1)&=\frac{49}{64}u^{-1/4}(x_1)u^2_{\nu}\\
&=49u^{7/4}(x_1)\left(\frac{u_\nu}{8u}\right)^2\\
&\le 49u^{7/4}(x_1)L^2.
\end{aligned}
\end{equation*}
By the definition of $G$ and $\phi(x)\equiv 1$ for $x\in B_{R/2}(\partial M)$, we indeed have
\begin{equation*}
\begin{aligned}
\sup\limits_{B_{R/2}(\partial M)}|\nabla h|^2(x)&=\sup\limits_{B_{R/2}(\partial M)}G(x)\\
&\le G(x_1)\\
&\le 49 L^2 u^{7/4}(x_1).
\end{aligned}
\end{equation*}
Since $h=u^{7/8}$, the above estimates gives that
\[
\sup\limits_{B_{R/2}(\partial M)}\frac{49}{64}u^{-1/4}(x)|\nabla u|^2\le 49 L^2\sup_{y\in B_R(\partial M)}u^{7/4}(y).
\]
Namely,
\[
\sup\limits_{B_{R/2}(\partial M)}|\nabla u|\le 8 L\sup_{y\in B_R(\partial M)}u(y)
\]
and the theorem follows.
\end{proof}



\begin{thebibliography}{99}
\bibitem{[BE]} Bakry, M. \'Emery, \emph{Diffusion hypercontractivitives}, in S\'{e}minaire de Probabilit\'{e}s XIX, 1983/1984,
in: Lecture Notes in Math., vol. 1123, Springer-Verlag, Berlin, 1985, 177-206.

\bibitem{Br} K. Brighton,  \emph{A Liouville-type theorem for smooth metric measure spaces}, J. Geom. Anal., \textbf{23} (2013), 562-570

\bibitem{Ch} R. Chen, \emph{Neumann eigenvalue estimate on a compact Riemannian manifold}, Proc. Amer. Math. Soc., \textbf{108} (1990), 961-970.

\bibitem{[CMZ]} X. Cheng, T. Mejia, D.-T. Zhou, \emph{Eigenvalue estimate and compactness for closed $f$-minimal surfaces}, Pacific J. Math., \textbf{271} (2014), 347-367.

\bibitem{DDW} H. T. Dung, N. T. Dung,  J.-Y. Wu, \emph{Sharp gradient estimates on weighted manifolds with compact boundary}, 	(2021), arXiv:2105.06185.

\bibitem{[Ham]} R. Hamilton, \emph{The formation of singularities in the Ricci flow}. Surveys in Differential Geometry,
International Press, Boston, vol. 2, (1995), 7-136.

\bibitem{KS20} K. Kunikawa, Y. Sakurai, \emph{Yau and Souplet-Zhang type gradient estimates on Riemannian manifolds with boundary under Dirichlet boundary condition}, arXiv:2012.09374.

\bibitem{[LY]}P. Li, S.-T. Yau, \emph{On the parabolic kernel of the Schrodinger
operator},  Acta Math., \textbf{156} (1986), 153-201.

\bibitem{[LD]}X.-D. Li, \emph{Liouville theorems for symmetric diffusion operators on complete Riemannian manifolds},
J. Math. Pure. Appl., \textbf{84} (2005), 1295--1361.

\bibitem{LV} J. Lott, C. Villani, \emph{Ricci curvature for metric-measure spaces via optimal transport},
 Ann. of Math., \textbf{169} (2009), 903-991.

\bibitem{[MD]} L. Ma,  S.-H. Du, \emph{Extension of Reilly formula with applications to eigenvalue estimates for drifting Laplacians}, C. R. Math. Acad. Sci. Paris, \textbf{348} (2010), 1203-1206.

\bibitem{Ol} X. Ramos Oliv\'e, \emph{Neumann Li-Yau gradient estimate under integral Ricci curvature bounds}, Proc. Amer. Math. Soc., \textbf{147} (2019), 411-426.

\bibitem{[P1]} G. Perelman, \emph{The entropy formula for the Ricci flow and its geometric applications},
(2002), arXiv:math.DG/0211159.

\bibitem{[P2]} G. Perelman, \emph{Ricci flow with surgery on three-manifolds}, (2003), arXiv:math.DG/0303109.

\bibitem{[P3]} G. Perelman, \emph{Finite extinction time for the solutions to the Ricci flow on certain three-manifolds},
 (2003), arXiv:math.DG/0307245.

\bibitem{Rei} R. C. Reilly, \emph{Applications of the Hessian operator in a Riemannian manifold}, Indiana Univ. Math. J.,
\textbf{26} (1977), 459-472.

\bibitem{Sa17} Y. Sakurai, \emph{Rigidity of manifolds with boundary under a lower Ricci curvature
bound}, Osaka J. Math., \textbf{54} (2017), 85-119.

\bibitem{Sakurai17} Y. Sakurai, \emph{Concentration of $1$-Lipschitz functions on manifolds with boundary with Dirichlet boundary condition}, (2017), arXiv:1712.04212v4.

\bibitem{Saku} Y. Sakurai, \emph{Rigidity of manifolds with boundary under a lower Bakry-\'Emery Ricci curvature bound}, Tohoku Math. J., \textbf{71} (2019), 69-109.

\bibitem{[SZ]}P. Souplet, Q S. Zhang,  \emph{Sharp gradient estimate and Yau's
Liouville theorem for the heat equation on noncompact manifolds},
Bull. London Math. Soc., \textbf{38} (2006), 1045-1053.

\bibitem{WaJ} J.-P. Wang, \emph{Global heat kernel estimates}, Pacific J. Math., \textbf{178} (1997), 377-398.

\bibitem{WZZ} L.-F. Wang, Z.-Y. Zhang, Y.-J. Zhou, \emph{Comparison theorems on smooth metric measure spaces with boundary}, Adv. Geom., \textbf{16} (2016), 401-411.

\bibitem{[WW]}G.-F. Wei, W. Wylie, \emph{Comparison geometry for the Bakry-Emery Ricci tensor}, J. Diff. Geom.,
\textbf{83} (2009), 377-405.

\bibitem{[WW1]} J.-Y. Wu, P. Wu, \emph{Heat kernel on smooth metric measure spaces with nonnegative curvature}, Math. Ann., \textbf{362} (2015), 717-742.

\bibitem{[WW2]}J.-Y. Wu, P. Wu, \emph{Heat kernel on smooth metric measure spaces and applications}, Math. Ann., \textbf{365 }(2016), 309-344.

\bibitem{Yau} S. T. Yau, \emph{Harmonic functions on complete Riemannian manifolds}, Comm. Pure Appl. Math., \textbf{28} (1975), 201-228.
\end{thebibliography}
\end{document}